\documentclass[12pt]{amsart}

\usepackage{amsmath,amssymb,amsthm}
\usepackage[bookmarksopen=true,
            bookmarksnumbered=true,
            colorlinks=true,
            linkcolor=blue,
            pdfauthor={Rachelle R. Bouchat, Huy T{\`a}i H{\`a}, and Augustine O'Keefe},
            pdftitle={Path Ideals of Rooted Trees and Their Graded Betti Numbers},
            pdfsubject={Combinatorial Commutative Algebra},
            pdfkeywords={edge ideal, path ideal, simple graph, regularity, linear strand, tree},
            pdfpagemode=UseOutlines,
            pdfproducer={Latex with hyperref},
            letterpaper=true,
            pdfcreator={latex->dvips->ps2pdf}]{hyperref}
\usepackage{graphicx,epsfig}
\usepackage{enumerate}
\usepackage{xypic}
  \input{xy}
  \xyoption{all}
\usepackage{url}

\numberwithin{equation}{section}

\swapnumbers
\theoremstyle{plain}
\newtheorem{theorem}{Theorem}[section]

\newtheorem{lemma}[theorem]{Lemma}

\newtheorem{corollary}[theorem]{Corollary}

\theoremstyle{definition}
\newtheorem{example}[theorem]{Example}
\newtheorem{remark}[theorem]{Remark}
\newtheorem{definition}[theorem]{Definition}

\DeclareMathOperator{\B}{\beta}

\DeclareMathOperator{\height}{height}

\DeclareMathOperator{\level}{level}
\DeclareMathOperator{\pd}{pd}
\DeclareMathOperator{\reg}{reg}

\begin{document}

\author{Rachelle R. Bouchat}
\address{Department of Mathematics\\
Slippery Rock University\\
1 Morrow Way, Slippery Rock, PA 16057\\
\url{http://academics.sru.edu/math/MATH/RRB/index.html}}
\email{rachelle.bouchat@sru.edu}

\author{Huy T{\`a}i H{\`a}}
\address{Department of Mathematics\\
Tulane University\\
6823 St. Charles Ave., New Orleans, LA 70118\\
\url{www.math.tulane.edu/~tai/}}
\email{tai@math.tulane.edu}

\author{Augustine O'Keefe}
\address{Department of Mathematics\\
Tulane University\\
6823 St. Charles Ave., New Orleans, LA 70118}
\email{aokeefe@tulane.edu}

\title{Path Ideals of Rooted Trees and Their Graded Betti Numbers}
\begin{abstract}
  Let $\Gamma$ be a rooted (and directed) tree, and let $t$ be a positive integer. The path ideal $I_t(\Gamma)$ is generated by monomials that correspond to directed paths of length $(t-1)$ in $\Gamma$. In this paper, we study algebraic properties and invariants of $I_t(\Gamma)$. We give a recursive formula to compute the graded Betti numbers of $I_t(\Gamma)$ in terms of path ideals of subtrees. We also give a general bound for the regularity, explicitly compute the linear strand, and investigate when $I_t(\Gamma)$ has a linear resolution.
\end{abstract}

\maketitle


\section{Introduction}\label{introSection}

The construction of edge ideals associated to (hyper)graphs (cf. \cite{HVT2, villarreal}) provides a viewpoint complementary to the Stanley-Reisner correspondence in the study of monomial ideals. Edge ideals also provide a framework to study (hyper)graph theoretic questions from an algebraic perspective. Let $\Gamma = (V,E)$ be a finite, simple graph over the vertex set $V = \{x_1, \dots, x_n\}$. Let $k$ be any field and identify the vertices in $V$ with the variables in the polynomial ring $S = k[x_1, \dots, x_n]$. The edge ideal of $\Gamma$ is generated by monomials of the form $x_ix_j$, where $e = \{x_i, x_j\}$ is an edge in $\Gamma$. Note that an edge can be viewed as a path of length 1. Thus, for a given positive integer $t$, a more general construction is obtained by considering monomials corresponding to paths of length $(t-1)$ in $\Gamma$. This is the {\it path ideal} construction.

Path ideals were first introduced by Conca and De Negri in \cite{concaDeNegri}, and their algebraic properties have been investigated by various authors in the literature (cf. \cite{brumatti, concaDeNegri, heVanTuyl, restuccia}). In this paper, we shall study path ideals of {\it rooted trees}. Recall that a {\it tree} is a graph in which there exists a unique path between every pair of distinct vertices; a {\it rooted tree} is a tree together with a fixed vertex called the {\it root}. In particular, in a rooted tree there exists a unique path from the root to any given vertex. We can also view a rooted tree as a {\it directed} graph by assigning to each edge the direction that goes ``away'' from the root. {\it Throughout this paper, a rooted tree will always be viewed as a directed, rooted tree in this sense}. If $\{x_i,x_j\}$ is an edge in a rooted tree $\Gamma$, then we write $(x_i,x_j)$ for the ``directed'' edge whose direction is from $x_i$ to $x_j$. The path ideal of a rooted tree is defined in precise form as follows.

\begin{definition} Let $t \ge 1$ be a given integer, and let $\Gamma$ be a rooted tree with vertex set $V=\{x_1,\ldots ,x_n\}$.
\begin{enumerate}
    \item A directed \emph{path} of length $(t-1)$ is a sequence of distinct vertices $x_{i_1}, \dots, x_{i_t}$, in which $(x_{i_j}, x_{i_{j+1}})$ is the directed edge from $x_{i_j}$ to $x_{i_{j+1}}$ for any $j = 1, \dots, t-1$.
    \item The \emph{path ideal} of length $(t-1)$ associated to $\Gamma$ is the monomial ideal
        \[I_t(\Gamma):=(x_{i_1}\cdots x_{i_t}\hspace{2mm}|\hspace{2mm}x_{i_1},\ldots,x_{i_t}\mbox{ is a path in } \Gamma)\subset S=k[x_1,\ldots ,x_n].\]
    In particular, when $t = 1$, $I_1(\Gamma) = (x_1, \dots, x_n)$ is the maximal homogeneous ideal, which is well understood.  Hence, all of our results in this paper will be for path ideals of length at least 1 (i.e., $t \ge 2$).
\end{enumerate}
\end{definition}

Due to the correspondence between paths and monomials we shall often abuse notation and use $x_{i_1} \cdots x_{i_t}$ to denote both the monomial $x_{i_1} \cdots x_{i_t}$ in $k[x_1, \dots, x_n]$ and the path $x_{i_1}, \dots, x_{i_t}$ in $\Gamma$.

A {\it rooted forest} is a disjoint union of rooted trees. For a rooted forest $\Delta$, we define the path ideal $I_t(\Delta)$ to be the sum of the path ideals of the connected components of $\Delta$.

A path ideal $I_t(\Gamma)$ is a squarefree monomial ideal, so it can also be realized as the edge ideal of a hypergraph or the Stanley-Reisner ideal of a simplicial complex. Given a path ideal, the corresponding hypergraph and simplicial complex are in general very complicated. The goal of this paper is then to investigate algebraic properties and invariants of a path ideal $I_t(\Gamma)$ via the combinatorial structures of the rooted tree $\Gamma$. We are interested in invariants associated to the minimal free resolution of $I_t(\Gamma)$, namely, the graded Betti numbers, the Castelnuovo-Mumford regularity, and the projective dimension.

We now provide an overview of the structure of the paper and our results. In Section~\ref{treeSection}, we recall some useful notation and terminology, and prove our first main result; here we give a recursive formula to compute the graded Betti numbers of path ideals (Theorem \ref{nocancellation} and Remark \ref{rmk.recursive}). Section~\ref{sec.regularity} is devoted to studying the Castelnuovo-Mumford regularity of path ideals. The main result of this section, Theorem \ref{thm.reg}, provides a general bound for the regularity of a path ideal in terms of the number of leaves and the number of pairwise disjoint paths of length $t$ in the tree. In Section \ref{sec.linear}, we study the linear strand of $I_t(\Gamma)$ and classify all rooted trees $\Gamma$ for which $I_t(\Gamma)$ has a linear resolution. Our first result of this section, Theorem~\ref{thm.linearstrand}, gives a precise formula for graded Betti numbers $\beta_{i,i+t}(I_t(\Gamma))$ on the linear strand of $I_t(\Gamma)$. Our next result in this section, Theorem \ref{thm.Ntp}, shows that $I_t(\Gamma)$ has a linear resolution if and only if it has linear first syzygies; this is the case if any only if $\Gamma$ belongs to a special class of rooted trees, which we will call broom graphs. In Section \ref{line_graphs}, we restrict our attention to rooted trees occurring as path graphs. For a path graph $\Gamma$, in Theorem~\ref{thm.nonzero}, we characterize which graded Betti numbers of $I_t(\Gamma)$ are nonzero. As a consequence, we compute the regularity of $S/I_t(\Gamma)$ explicitly in Corollary \ref{thm.reglinegraph}. We also recover He and Van Tuyl's formula for the projective dimension of $S/I_t(\Gamma)$ in this case (Corollary~\ref{pdlinegraph}).

\noindent{\bf Acknowledgement.} This project started when the first author visited the other authors at Tulane University. The authors wish to thank Tulane University for its hospitality. The authors would also like to thank Adam Van Tuyl for stimulating discussions and suggestions, and to thank the two anonymous referees for many useful comments making the paper more readable. The second author acknowledges support from the Board of Regents grant LEQSF(2007-10)-RD-A-30.


\section{Path Ideals and Graded Betti Numbers}\label{treeSection}

From this point forward, $\Gamma$ will denote a rooted tree (also viewed as a directed tree) with vertex set $V = \{x_1, \dots, x_n\}$, $k$ will denote a field of arbitrary characteristic, and $t$ will denote a given positive integer. Then $S = k[x_1, \dots, x_n]$ will denote the corresponding polynomial ring.

\vspace{0.20cm}
\noindent{\bf Induced Subgraphs and Examples.} We will now introduce some combinatorial terminology and provide examples of path ideals.

\begin{definition} Let $\Gamma$ be a rooted tree with root $x$. For a given vertex $y$ in $\Gamma$, the {\it level} of $y$, denoted $\level(y)$, is defined to be the length of the unique path from $x$ to $y$. The {\it height} of $\Gamma$, denoted $\height(\Gamma)$, is the maximal level of vertices in $\Gamma$.
\end{definition}

Sometimes we will need to consider rooted forests. The level of a vertex $y$ in a rooted forest $\Delta$ is defined to be the level of $y$ inside the connected component of $\Delta$ containing $y$. The height of a rooted forest $\Delta$ is defined to be the largest height of its connected components.

\begin{example}
Consider the following rooted tree $\Gamma$.
    \begin{center}
    \begin{minipage}[h]{3.5in}
            \includegraphics{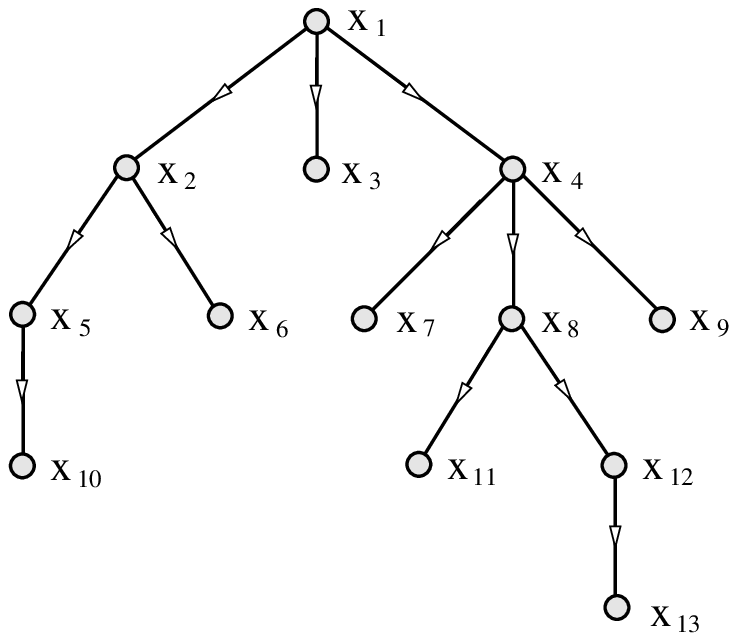}
    \end{minipage}
    \end{center}
The edges in $\Gamma$ are given directions that go ``away'' from the root, making $\Gamma$ a directed tree. For instance, there is a unique path $x_1, x_2$ going from the root $x_1$ to the vertex $x_2$, and a unique path $x_1, x_2, x_6$ going from the root $x_1$ to the vertex $x_6$; and so, the direction of the edge $\{x_2, x_6\}$ is from $x_2$ to $x_6$. It can also be seen that the highest level in $\Gamma$ is 4 ($\level(x_{13}) = 4)$), so $\height(\Gamma) = 4$.

\noindent For $t = 2, 3, 4$, and $5$, we have the following path ideals associated to $\Gamma$:
    \[\begin{array}{rcl}
        I_2(\Gamma) & = & (x_1x_2,x_1x_3,x_1x_4,x_2x_5,x_2x_6,x_4x_7,x_4x_8,x_4x_9,x_5x_{10}, x_8x_{11},x_8x_{12},x_{12}x_{13})\\
        I_3(\Gamma) & = & (x_1x_2x_5,x_1x_2x_6,x_1x_4x_7,x_1x_4x_8,x_1x_4x_9,x_2x_5x_{10}, x_4x_8x_{11},x_4x_8x_{12},x_8x_{12}x_{13})\\
        I_4(\Gamma) & = & (x_1x_2x_5x_{10},x_1x_4x_8x_{11},x_1x_4x_8x_{12},x_4x_8x_{12}x_{13})\\
        I_5(\Gamma) & = & (x_1x_4x_8x_{12}x_{13}).
    \end{array}\]
Notice that the path ideals of $\Gamma$ depend on the choice of the root of $\Gamma$.
\end{example}

\begin{definition} Let $\Gamma$ be a rooted tree, and let $x$ be a vertex in $\Gamma$.
\begin{enumerate}
\item A vertex $z$ in $\Gamma$ is the {\it parent} of $x$ if and only if  $(z,x)$ is a directed edge in $\Gamma$. A vertex $y$ is called a {\it child} of $x$ if $(x,y)$ is a directed edge in $\Gamma$.
\item A vertex $z \not= x$ is an {\it ancestor} of $x$ if there is a path from $z$ to $x$. A vertex $y \not= x$ is a {\it descendant} of $x$ if there is a path from $x$ to $y$.
\item The vertex $x$ is called a {\it leaf} of $\Gamma$ if $x$ has no descendants.
\item The vertex $x$ is called the {\it root} of $\Gamma$ if $x$ has no ancestors.
\item The {\it degree} of a vertex $x$ in $\Gamma$, denoted by $\deg_\Gamma(x)$, is the number of edges in $\Gamma$ incident to $x$.
\end{enumerate}
\end{definition}

\begin{definition} \quad
\begin{enumerate}
\item Let $G$ be a finite simple graph. A subgraph $H$ of $G$ is called an {\it induced subgraph} if for every pair of vertices $x, y$ in $H$ the following condition holds: if $\{x,y\}$ is an edge in $G$, then it is also an edge in $H$.
\item Let $\Gamma$ be a rooted tree. An {\it induced subtree} (or {\it forest}) of $\Gamma$ is a directed subtree (or forest) that is also an induced subgraph of $\Gamma$.
\item Let $\Gamma$ be a rooted tree and let $x$ be a vertex in $\Gamma$. The {\it induced subtree rooted at $x$} of $\Gamma$ is the induced subtree of $\Gamma$ over the vertex set $\{x\} \cup \{y ~|~ y \text{ is a descendant of } x\}$ (with $x$ considered as its root).
\end{enumerate}
\end{definition}

\noindent{\bf Notation.} Let $\Gamma$ be a rooted tree, and let $P$ be a collection of vertices in $\Gamma$. We shall denote by $\Gamma \backslash P$ the induced subforest of $\Gamma$ obtained by removing the vertices in $P$ and the edges incident to these vertices. If $P$ consists of a single element $x$, then we write $\Gamma \backslash x$ for $\Gamma \backslash \{x\}$.

\vspace{0.20cm}
\noindent{\bf Minimal Free Resolutions.}
Let $S = k[x_1, \dots, x_n]$, and let $M$ be a finitely generated graded $S$-module. Associated to $M$ is a \emph{minimal free resolution}, which is a finite complex of the form
    \[0\rightarrow\bigoplus_j S(-j)^{\beta_{p,j}(M)}\stackrel{\delta_p}{\longrightarrow}\bigoplus_j S(-j)^{\beta_{p-1,j}(M)}\stackrel{\delta_{p-1}}{\longrightarrow} \cdots\stackrel{\delta_1}{\longrightarrow}\bigoplus_j S(-j)^{\beta_{0,j}(M)}\rightarrow M\rightarrow 0\]
where the maps $\delta_i$ are exact and where $S(-j)$ denotes the translation of $S$ obtained by shifting the degrees of elements of $S$ by $j$. The numbers $\beta_{i,j}(M)$ are called the \emph{graded Betti numbers} of $M$, and they provide the number of minimal generators of degree $j$ occuring in the $i$th-syzygy module of $M$.


If $M$ is generated in degree $t$, then the {\it linear strand} of $M$ is given by the Betti numbers $\beta_{i,i+t}(M)$, for $i > 0$. In this case, $M$ is said to have linear first syzygies if $\beta_{1,j}(M) = 0$ for all $j \not= 1+t$; and more generally, $M$ is said to have a {\it linear resolution} if $\beta_{i,j}(M) = 0$ for all $i > 0$ and $j \not= i+t$.

We are interested in the following two invariants that measure the ``size'' of the minimal free resolution.

\begin{definition} Let $S$ and $M$ be as above.
\begin{enumerate}
\item The \emph{projective dimension} of $M$, denoted by $\pd(M)$, is the length of the minimal free resolution associated to $M$.
\item The \emph{Castelnuovo-Mumford regularity} (or simply, \emph{regularity}), denoted by $\reg(M)$, is a measure of the width of the minimal free resolution of $M$ and is defined as
    \[\reg(M):=\max\{j-i\hspace{2mm}|\hspace{2mm}\beta_{i,j}(M)\neq 0\}.\]
\end{enumerate}
\end{definition}

\begin{lemma} \label{lem.ringextension}
Let $S = k[x_1, \dots, x_n]$, and let $M$ be a graded $S$-module. Let $y_1, \dots, y_m$ be indeterminates, and denote by $R$ the polynomial ring $k[x_1, \dots, x_n, y_1, \dots, x_m]$. Then
$$\reg(M) = \reg(M \otimes_S R),$$
where the second regularity is computed for the $R$-module $M \otimes_S R$.
\end{lemma}

\begin{proof} It is clear that the ring extension $S \rightarrow R$ is flat. Thus, tensoring with $R$ maps a minimal free resolution to a minimal free resolution. The result now follows from the definition of regularity.
\end{proof}

Lemma \ref{lem.ringextension} allows us to look at extensions of ideals in rings with more variables when discussing regularity. For instance, if $\Gamma$ is a rooted tree corresponding to the polynomial ring $R$ and $\Delta$ is an induced, rooted subtree of $\Gamma$ corresponding to the polynomial ring $S$ (i.e.\ the variables in $R$ correspond to the vertices in $\Gamma$, and the variables in $S$ correspond to the vertices in $\Delta$), then we can abuse notation and write $I_t(\Delta)$ for both the path ideal of $\Delta$ in $S$ and also for its extension in the bigger ring $R$ when discussing regularity.

\vspace{0.20cm}
\noindent{\bf Mapping Cone Decomposition of Path Ideals.}
To study the minimal free resolutions of the quotient rings $S/I_t(\Gamma)$, we provide an inductive construction of the path ideals via the mapping cone construction.  This construction is a generalization of the method provided for edge ideals in \cite{bouchat}. This method will allow the decomposition of a given path ideal into a collection of simpler path ideals corresponding to smaller trees.  

Given a short exact sequence
    \[0\longrightarrow M_1\longrightarrow M_2\longrightarrow M_3\longrightarrow 0\]
where $M_1$, $M_2$, and $M_3$ are graded $S$-modules, the mapping cone is a method to construct a free resolution for $M_3$ knowing free resolutions of $M_1$ and $M_2$ (for more details on the mapping cone construction we refer the reader to \cite{Weibel}).
In general, given {\it minimal} free resolutions for $M_1$ and $M_2$, the mapping cone construction does not necessarily give a {\it minimal} free resolution of $M_3$. However, in the case of path ideals, we shall show that the mapping cone construction does indeed provide a minimal free resolution for a particular short exact sequence.

\begin{theorem}\label{nocancellation}
    Let $\Gamma$ be a rooted tree with vertex set $V=\{x_1,\ldots ,x_n\}$ and height $h \ge t-1$.  Let $x_{i_t}$ denote a leaf of $\Gamma$ of level at least $(t-1)$.  Then by letting $x_{i_1},\ldots ,x_{i_t}$ denote the path terminating at $x_{i_t}$, the mapping cone procedure applied to the sequence
            \[0\rightarrow \big(S\big/I_t(\Gamma \backslash x_{i_{t}}): (x_{i_1}\cdots x_{i_t})\big)(-t)\stackrel{x_{i_1}\cdots x_{i_t}}{\longrightarrow} S/I_t(\Gamma \backslash x_{i_{t}}) \longrightarrow S/I_t(\Gamma) \rightarrow 0\]
    provides a minimal free resolution of $S/I_t(\Gamma)$. In particular, for any $i$ and $j$, we have
        \[\B_{i,j}(S/I_t(\Gamma))=\B_{i,j}(S/I_t(\Gamma \backslash x_{i_{t}})) + \B_{i-1,j-t}\big(S\big/I_t(\Gamma \backslash x_{i_{t}}):(x_{i_1}\cdots x_{i_t})\big).\]
\end{theorem}

\begin{proof}
    Since $x_{i_t}$ does not divide a minimal generator of $I_t(\Gamma \backslash x_{i_{t}})$,
        \[I_t(\Gamma \backslash x_{i_{t}}): (x_{i_1}\cdots x_{i_t}) = I_t(\Gamma \backslash x_{i_{t}}) : (x_{i_1}\cdots x_{i_{t-1}}).\]
    However, this implies that the exact sequence
    \begin{align}
        0\longrightarrow \big(S\big/I_t(\Gamma \backslash x_{i_{t}}): (x_{i_1}\cdots x_{i_t})\big)(-t) \stackrel{x_{i_1}\cdots x_{i_t}}{\longrightarrow} S/I_t(\Gamma \backslash x_{i_{t}}) \longrightarrow S/I_t(\Gamma) \rightarrow 0 \label{eq:sequence}
    \end{align}
    factors as
       \begin{small}\begin{align}\xymatrix@R=20pt{
            0\ar[r] & \big(S\big/I_t(\Gamma \backslash x_{i_{t}}): (x_{i_1}\cdots x_{i_t})\big)(-t) \ar[r]^{\mbox{\hspace{.6in}}x_{i_1}\cdots x_{i_t}}\ar[dd]^{x_{i_t}} & S/I_t(\Gamma \backslash x_{i_{t}}) \ar[r] & S/I_t(\Gamma) \ar[r] & 0.\\
            & & & &\\
            & \big(S\big/I_t(\Gamma \backslash x_{i_{t}}) :((x_{i_1}\cdots x_{i_{t-1}}))\big)(-t+1) \ar[uur]^{x_{i_1}\cdots x_{i_{t-1}}} & & &}\label{eq:factor}
        \end{align}\end{small}
            
Let
     \begin{align}
     & 0 \longrightarrow \dots \stackrel{\phi_2}{\longrightarrow} F_1 \stackrel{\phi_1}{\longrightarrow} F_0=S \stackrel{\phi_0}{\longrightarrow} S\big/I_t(\Gamma \backslash x_{i_{t}}): (x_{i_1}\cdots x_{i_t}) \longrightarrow 0, \mbox{ and } \label{eq:mfr1} \\
     & 0 \longrightarrow \dots \stackrel{\psi_2}{\longrightarrow} G_1 \stackrel{\psi_1}{\longrightarrow} G_0=S \stackrel{\psi_0}{\longrightarrow} S/I_t(\Gamma \backslash x_{i_{t}}) \longrightarrow 0 \label{eq:mfr2}
     \end{align}
 \noindent be minimal free resolutions of $\big(S\big/I_t(\Gamma \backslash x_{i_{t}}): (x_{i_1}\cdots x_{i_t})\big)$ and $S/I_t(\Gamma \backslash x_{i_{t}})$ respectively.  Then the mapping cone construction applied to the short exact sequence \eqref{eq:sequence} provides a free resolution of $S/I_t(\Gamma)$ given by
   $$0\longrightarrow \dots \stackrel{\sigma_3}{\longrightarrow} G_2\bigoplus F_1(-t) \stackrel{\sigma_2}{\longrightarrow} G_1\bigoplus S(-t) \stackrel{\sigma_1}{\longrightarrow} S \stackrel{\sigma_0}{\longrightarrow} S/I_t(\Gamma) \longrightarrow 0,$$     
  \noindent where the maps $\sigma_i$ are defined by $\sigma_1=[\psi_1-\delta_0]$ and 
    \begin{align}\sigma_i =\left[ \begin{array}{cc} \psi_i & (-1)^i\delta_{i-1}\\ 0 & \phi_{i-1}\end{array}\right]\mbox{\hspace{.25in}for }i>1, \label{eq:sigma} \end{align}
($\delta_i:F_i(-t) \rightarrow G_i$ are resulted from the homomorphism $\big(S\big/I_t(\Gamma \backslash x_{i_{t}}): (x_{i_1}\cdots x_{i_t})\big)(-t) \stackrel{x_{i_1}\cdots x_{i_t}}{\longrightarrow} S/I_t(\Gamma \backslash x_{i_{t}})$).  

From the factorization given in \eqref{eq:factor}, the entries of the matrix of $\delta_i$ are not units. Furthermore, since \eqref{eq:mfr1} and \eqref{eq:mfr2} are minimal free resolutions, the matrix representation of $\sigma_i$ in \eqref{eq:sigma} cannot contain units.  Therefore, the mapping cone construction applied to \eqref{eq:sequence} provides a minimal free resolution of $S/I_t(\Gamma)$.  In particular, this implies that
        \[\B_{i,j}(S/I_t(\Gamma))=\B_{i,j}(S/I_t(\Gamma \backslash x_{i_{t}}))+ \B_{i-1,j-t}\big(S\big/I_t(\Gamma \backslash x_{i_{t}}) :(x_{i_1}\cdots x_{i_t})\big)\]
    for all $i$ and $j$.
\end{proof}

Theorem \ref{nocancellation} provides an inductive method to study algebraic properties of $I_t(\Gamma)$ as the colon ideal $I_t(\Gamma \backslash x_{i_{t}}):(x_{i_1}\cdots x_{i_t})$ can be realized as a disjoint union of path ideals of varying lengths.

\begin{lemma} \label{lem.subtree}
Let $\Gamma$ be a rooted tree of height $h \ge t-1$, let $x_{i_t}$ be a leaf at the highest level in $\Gamma$, and let $x_{i_1}, \dots, x_{i_t}$ be the unique path of length $(t-1)$ terminating at $x_{i_t}$. Let $x_{i_0}$ be the only parent of $x_{i_1}$, if it exists. For $j = 0, \dots, t$, let $\Gamma_j$ be the induced subtree of $\Gamma$ rooted at $x_{i_j}$, and let $\Delta_j = \Gamma_j \backslash \Gamma_{j+1}$ for $j = 0, \dots, t-1$. Then
\begin{align*}
I_t(\Gamma \backslash x_{i_t}) : (x_{i_1} \cdots x_{i_t}) & = I_t(\Gamma \backslash \{x_{i_0}, \dots, x_{i_t}\}) + (x_{i_0}) + \sum_{j=0}^{t-1} I_{t-j}(\Delta_j \backslash \{x_{i_0}, \dots, x_{i_t}\}). \\
\end{align*}
\end{lemma}

\begin{proof} Let $G$ be the set of minimal generators of $I_t(\Gamma \backslash x_{i_t})$, i.e.\ elements in $G$ corresponding to paths of length $(t-1)$ in $\Gamma \backslash x_{i_t}$. Clearly,
$$I_t(\Gamma \backslash x_{i_t}) : (x_{i_1} \cdots x_{i_t}) = \sum_{Q \in G} (Q) : (x_{i_1} \cdots x_{i_t}).$$

Observe first that $Q_1 = x_{i_0} x_{i_1} \cdots x_{i_{t-1}}$ is a path of length $(t-1)$ in $\Gamma \backslash x_{i_t}$, and $(Q_1) : (x_{i_1} \cdots x_{i_{t-1}}) = (x_{i_0})$. Consider a path $Q$ of length $(t-1)$ in $\Gamma \backslash x_{i_t}$ that does not contain $x_{i_0}$. There are three possibilities for $Q$.

\noindent{\bf Case 1:} $Q$ contains none of the vertices in $\{x_{i_0}, \dots, x_{i_t}\}$, and $Q$ is not a path in the induced subtree rooted at $x_{i_0}$. This is the case if and only if $(Q) : (x_{i_1} \cdots x_{i_t}) = (Q) \subseteq I_t(\Gamma \backslash \{x_{i_0}, \dots, x_{i_t}\})$.

\noindent{\bf Case 2:} $Q$ contains none of the vertices in $\{x_{i_0}, \dots, x_{i_t}\}$, and $Q$ is a path in the induced subtree rooted at $x_{i_0}$. This is the case if and only if $Q$ is a path of length $(t-1)$ in the rooted forest $\Gamma_0 \backslash \{x_{i_0}, \dots, x_{i_t}\}$.

\noindent{\bf Case 3:} $Q$ contains some but not all of the vertices $\{x_{i_1}, \dots, x_{i_{t-1}}\}$, and $Q$ does not contain $x_{i_0}$. Let $s$ be the largest index such that $x_{i_s}$ is in $Q$. Since $x_{i_t}$ is a leaf of highest level in $\Gamma$, $Q$ can contain at most $t-s$ descendants of $x_{i_s}$. This implies that $Q$ must contain all the vertices $x_{i_1}, \dots, x_{i_s}$. This is the case only if $Q \backslash \{x_{i_1}, \dots, x_{i_s}\}$ is a path of length $t-s-1$ in $\Gamma_s \backslash \{x_{i_0}, \dots, x_{i_t}\}$. Furthermore, because $x_{i_t}$ is of highest level in $\Gamma$, any path of length $(t-s-1)$ in $\Gamma_s$ must be from a child of $x_{i_s}$ (other than $x_{i_{s+1}}$) to a leaf in $\Gamma_s$. Thus, Case 3 appears if and only if $(Q) : (x_{i_1} \cdots x_{i_t}) \subseteq I_{t-s}(\Delta_s \backslash \{x_{i_0}, \dots, x_{i_t}\}).$
\end{proof}

\begin{remark} \label{rmk.recursive}
Note that in Lemma \ref{lem.subtree}, $I_t(\Gamma \backslash \{x_{i_0}, \dots, x_{i_t}\}) = I_t(\Gamma \backslash \Gamma_0)$ (or $I_t(\Gamma \backslash \Gamma_1)$ if $x_{i_0}$ does not exist) since $x_{i_t}$ is a leaf at the highest level in $\Gamma$.
Observe further that the minimal generators of the ideals $I_t(\Gamma \backslash \{x_{i_0}, \dots, x_{i_t}\}), (x_{i_0})$, and $I_{t-j}(\Delta_j \backslash \{x_{i_0}, \dots, x_{i_t}\})$ involve pairwise disjoint sets of vertices. Thus, the minimal free resolution of $S/[I_t(\Gamma \backslash \{x_{i_0}, \dots, x_{i_t}\}) + (x_{i_0}) + \sum_{j=0}^{t-1} I_{t-j}(\Delta_j \backslash \{x_{i_0}, \dots, x_{i_t}\})]$ is obtained by taking the tensor product of the minimal free resolutions of $S/I_t(\Gamma \backslash \{x_{i_0}, \dots, x_{i_t}\})$, $S/(x_{i_0})$, and $S/I_{t-j}(\Delta_j \backslash \{x_{i_0}, \dots, x_{i_t}\})$ for $j=0,\ldots,t-1$. Together with Theorem \ref{nocancellation}, this gives a recursive formula to compute the graded Betti numbers of $I_t(\Gamma)$. In particular, the graded Betti numbers of $I_t(\Gamma)$ do not depend on the characteristic of the ground field $k$. This fact was proved in \cite[Theorem 3.1]{heVanTuyl}. It is also a corollary of a more general recursive formula for the graded Betti numbers of simplicial forests given in \cite[Theorem 5.8]{HVT1}.
\end{remark}


\section{Regularity of Path Ideals} \label{sec.regularity}

In this section, we give a bound for the regularity of $I_t(\Gamma)$. From the Alexander duality (cf. \cite[Theorem 5.59]{MillerSturmfels2004}), one obtains the following trivial bound $\reg(S/I_t(\Gamma)) = \reg(I_t(\Gamma)) - 1 = \pd(S/I_t(\Gamma)^\vee) - 1 \le n-1$, where $I_t(\Gamma)^\vee$ is the Alexander dual of $I_t(\Gamma)$. We are seeking a bound for $\reg(S/I_t(\Gamma))$ that is, in general, better than $n-1$. Our bound will be based on the number of leaves and the number of pairwise disjoint paths of length $t$ in $\Gamma$.

\begin{definition}
Let $\Gamma$ be a rooted tree. We define $l_t(\Gamma)$ to be the number of leaves in $\Gamma$ whose level is at least $t-1$ and $p_t(\Gamma)$ to be the maximal number of pairwise disjoint paths of length $t$ in $\Gamma$ (i.e., $p_t(\Gamma) = \max \{|D| ~\big|~ D \text{ is a set of disjoint paths of length $t$ in $\Gamma$}\}$). Note that, in general, $tl_t(\Gamma) \ll n$ and $tp_t(\Gamma) \ll n$.
\end{definition}

In the next few corollaries, $\Gamma$ will denote a rooted tree of height $h \ge t-1$, and $x_{i_t}$ will denote a leaf of highest level in $\Gamma$. Let $x_{i_1}, \dots, x_{i_t}$ be the unique path of length $(t-1)$ terminating at $x_{i_t}$, and let $x_{i_0}$ be the parent of $x_{i_1}$ (if it exists). Set $P = \{x_{i_0}, \dots, x_{i_t}\}$ (or $\{x_{i_1}, \dots, x_{i_t}\}$ if $x_{i_0}$ does not exist).  Furthermore for $j=0,\ldots,t-1$, let $\Gamma_j$ be the induced subtree of $\Gamma$ rooted at $x_{i_j}$, and let $\Delta_j = \Gamma_j \backslash \Gamma_{j+1}$.

\begin{corollary} \label{cor.reg}
We have
$$\reg\big(S\big/\big(I_t(\Gamma \backslash x_{i_t}): (x_{i_1} \cdots x_{i_t})\big)\big) = \reg(S/I_t(\Gamma \backslash P)) + \sum_{j=0}^{t-1} \reg\big(S\big/I_{t-j}(\Delta_j \backslash P)\big).$$
\end{corollary}

\begin{proof} It is easy to see that for $j=0,\ldots,t-1$ the minimal generators of the ideals $I_t(\Gamma \backslash P)$, $(x_{i_0})$, and $I_{t-j}(\Delta_j \backslash P)$ involve pairwise disjoint sets of vertices. Thus, by Lemma \ref{lem.subtree}, the minimal free resolution of $S\big/\big(I_t(\Gamma \backslash x_{i_t}): (x_{i_1} \cdots x_{i_t})\big)$ is given by the tensor product of the minimal free resolution of $S/I_t(\Gamma \backslash P)$, $S/(x_{i_0})$, and $S/I_{t-j}(\Delta_j \backslash P)$ for $j=0,\ldots,t-1$. This implies that
\begin{align*}
\reg\big(S\big/\big(I_t(\Gamma \backslash x_{i_t})  : (x_{i_1} \cdots x_{i_t})\big)\big) = &
\reg(S/I_t(\Gamma \backslash P)) + \reg(S/(x_{i_0})) + \sum_{j=0}^{t-1} \reg\big(S\big/I_{t-j}(\Delta_j \backslash P)\big).
\end{align*}
The conclusion now follows from the fact that $\reg(S/(x_{i_0})) = 0.$
\end{proof}

\begin{corollary} \label{cor.smalltree}
We have
$$\reg(S/I_t(\Gamma)) = \max \{\reg(S/I_t(\Gamma \backslash x_{i_t})), \reg(S/I_t(\Gamma \backslash P)) + \sum_{j=0}^{t-1} \reg(S/I_{t-j}(\Delta_j \backslash P)) + (t-1)\}.$$
In particular, by considering $\Gamma_0$ in place of $\Gamma$ we have
$$\reg(S/I_t(\Gamma_0)) = \max \{\reg(S/I_t(\Gamma_0 \backslash x_{i_t})), \sum_{j=0}^{t-1} \reg\big(S\big/I_{t-j}(\Delta_j \backslash P)\big) + (t-1) \}.$$
If $x_{i_0}$ does not exist, then
$$\reg(S/I_t(\Gamma_1)) = \max \{\reg(S/I_t(\Gamma_1 \backslash x_{i_t})), \sum_{j=1}^{t-1} \reg\big(S\big/I_{t-j}(\Delta_j \backslash P)\big) + (t-1) \}.$$
\end{corollary}

\begin{proof} It follows from Theorem \ref{nocancellation} that
$$\reg(S/I_t(\Gamma)) = \max\{ \reg(S/I_t(\Gamma \backslash x_{i_t})), \reg(S/I_t(\Gamma \backslash x_{i_t}): (x_{i_1} \cdots x_{i_t})) + (t-1)\}.$$
The first conclusion follows by applying Corollary \ref{cor.reg}. The second conclusion follows by observing that since $x_{i_t}$ is a leaf at the highest level, we have $I_t(\Gamma_0 \backslash P) = (0)$ (or $I_t(\Gamma_1 \backslash P) = (0)$ if $x_{i_0}$ does not exist).
\end{proof}

We are ready to prove our next theorem.

\begin{theorem} \label{thm.reg}
Let $\Gamma$ be a rooted tree over the vertex set $V = \{x_1, \dots, x_n\}$. Then
$$\reg(S/I_t(\Gamma)) \le (t-1)[l_t(\Gamma) + p_t(\Gamma)].$$
\end{theorem}

\begin{proof} We shall use induction on both $t$ and $n$. For $t = 1$, the ideal $I_t(\Gamma)$ is the maximal homogeneous ideal of $S = k[x_1, \dots, x_n]$, and the assertion is clearly true. Assume that $t \ge 2$. The assertion is also true if $n \le t$, so we may assume that $n > t$.

Let $h = \height(\Gamma)$. Observe that if $h < t-1$ then $I_t(\Gamma) = (0)$, making the assertion vacuous. We shall assume that $h \ge t-1$. Consider first the case when $h = t-1$. In this case, any path of length $(t-1)$ in $\Gamma$ must be from the root to a leaf (at level $(t-1)$) of $\Gamma$ and so $p_t(\Gamma) = 0$. Without loss of generality, assume that $x_1$ is the root of $\Gamma$. Then $I_t(\Gamma) = x_1I_{t-1}(\Gamma \backslash x_1)$ and $p_{t-1}(\Gamma \backslash x_1) = 0$.
By the induction hypothesis, we have
$$\reg(S/I_{t-1}(\Gamma \backslash x_1)) \le (t-2)[l_{t-1}(\Gamma \backslash x_1)+p_{t-1}(\Gamma \backslash x_1)] = (t-2)l_{t-1}(\Gamma \backslash x_1).$$
Observe further that $l_{t-1}(\Gamma \backslash x_1) = l_t(\Gamma) \ge 1$ (since $h = t-1$, $\Gamma$ must have at least a leaf at level $(t-1)$). Therefore, we have
$$\reg(S/I_t(\Gamma)) = \reg(S/I_{t-1}(\Gamma \backslash x_1)) + 1 \le (t-2)l_{t-1}(\Gamma \backslash x_1)+1 \le (t-1) l_t(\Gamma),$$
and the assertion is true.

Consider now the case when $h \ge t$. Let $x_t$ be a leaf at the highest level, and let $x_{i_0}, \dots, x_{i_t}$ be the unique path of length $t$ terminating at $x_{i_t}$. Let $P = \{x_{i_0}, \dots, x_{i_t}\}$, let $\Gamma_j$ be the induced subtree of $\Gamma$ rooted at $x_{i_j}$, and let $\Delta_j = \Gamma_j \backslash \Gamma_{j+1}$ for $j=0,\ldots,t-1$. It follows from Corollary \ref{cor.smalltree} that
\begin{align}
\reg(S/I_t(\Gamma)) = \max \{ & \reg(S/I_t(\Gamma \backslash x_{i_t})), \nonumber \\
& \reg(S/I_t(\Gamma \backslash P)) + \sum_{j=0}^{t-1} \reg(S/I_{t-j}(\Delta_j \backslash P)) + (t-1)\}. \label{eq.induction1}
\end{align}
Observe that $l_t(\Gamma \backslash x_{i_t}) \le l_t(\Gamma)$ and $p_t(\Gamma \backslash x_{i_t}) \le p_t(\Gamma)$. Thus, by the induction hypothesis, we have
$$\reg(S/I_t(\Gamma \backslash x_{i_t})) \le (t-1)[l_t(\Gamma \backslash x_{i_t}) + p_t(\Gamma \backslash x_{i_t})] \le (t-1)[l_t(\Gamma) + p_t(\Gamma)].$$
It can also be seen that $l_t(\Gamma \backslash P) = l_t(\Gamma \backslash \Gamma_0) \le l_t(\Gamma) - \l_t(\Gamma_0) + 1$ and $\sum_{j=0}^{t-1} l_{t-j}(\Delta_j \backslash P) = l_t(\Gamma_0) - 1$. Thus, by the induction hypothesis, we have
\begin{align*}
\reg(S/I_t(\Gamma \backslash P)) + \sum_{j=0}^{t-1} \reg(S/I_{t-j}(\Delta_j \backslash P)) + t-1 & \le (t-1)[l_t(\Gamma \backslash \Gamma_0) + p_t(\Gamma \backslash \Gamma_0) \\
& \quad + \sum_{j=0}^{t-1} (l_{t-j}(\Delta_j \backslash P) + p_{t-j}(\Delta_j \backslash P))] + (t-1) \\
& = (t-1) l_t(\Gamma) + (t-1)[p_t(\Gamma \backslash \Gamma_0) \\
& \quad + \sum_{j=0}^{t-1} p_{t-j}(\Delta_j \backslash P)] + (t-1).
\end{align*}
Moreover, $p_t(\Gamma \backslash \Gamma_0) \le p_t(\Gamma) - p_t(\Gamma_0)$, $p_t(\Gamma_0) = 1$, and $p_{t-j}(\Delta_j \backslash P) = 0$ (because $\height(\Delta_j \backslash P) \le t-j-1$ for any $j$). Hence,
\begin{align*}
\reg(S/I_t(\Gamma \backslash P)) + \sum_{j=0}^{t-1} \reg(S/I_{t-j}(\Delta \backslash P)) + t-1 
& \le (t-1) l_t(\Gamma) + (t-1)[p_t(\Gamma) - 1] + (t-1) \\
& \le (t-1)[l_t(\Gamma) + p_t(\Gamma)].
\end{align*}
The theorem is proved by the use of \eqref{eq.induction1}.
\end{proof}

\begin{remark} The bound in Theorem \ref{thm.reg} is sharp when $\Gamma$ is a disjoint union of paths of length $(t-1)$. For instance, if $\Gamma$ is a directed path $x_1 \rightarrow x_2 \rightarrow \dots \rightarrow x_t$ of length $(t-1)$, then $l_t(\Gamma) = 1$ and $p_t(\Gamma) = 0$.  Hence, $\reg(S/I_t(\Gamma)) = \reg(S/(x_1 \cdots x_t)) = t-1 = (t-1)[l_t(\Gamma) + p_t(\Gamma)].$
\end{remark}


\section{Linear Strand and Linear Resolution} \label{sec.linear}

In this section, we compute the linear strand of $I_t(\Gamma)$ for a rooted tree $\Gamma$, and classify all rooted trees $\Gamma$ for which $I_t(\Gamma)$ has a linear resolution.

We start by investigating the linear strand of $I_t(\Gamma)$. Note that $\beta_{0,t}(I_t(\Gamma))$ is just the number of paths of length $(t-1)$ in $\Gamma$.  Therefore we will be interested in $\beta_{i,i+t}(I_t(\Gamma))$ for $i \ge 1$. Note also that the path ideal $I_t(\Gamma)$ can be realized as the edge ideal of a {\it hyper-tree}. In \cite{HVT1}, the second author and Van Tuyl gave a formula for the linear strand of the edge ideal of any hyper-tree. However, the structure of the hyper-tree corresponding to $I_t(\Gamma)$ is quite complicated. We shall use the combinatorial data of $\Gamma$ to provide an explicit formula for the linear strand of $I_t(\Gamma)$.

\begin{lemma} \label{lem.ls1}
Let $\Gamma$ be a rooted tree of height $h \ge t-1$, let $x_{i_t}$ be a leaf at the highest level in $\Gamma$, and let $x_{i_1}, \dots, x_{i_t}$ be the unique path of length $(t-1)$ terminating at $x_{i_t}$. Then for $i > 0$,
$$\beta_{i,i}(S/(I_t(\Gamma \backslash x_{i_t}): (x_{i_1} \cdots x_{i_t}))) = \begin{cases}
{\displaystyle {\deg_\Gamma(x_{i_{t-1}})-2 \choose i}} & \text{if  } h = t-1 \text{ and } t \not= 2 \\
{\displaystyle {\deg_\Gamma(x_{i_{t-1}})-1 \choose i}} & \text{if  } h > t-1 \text{ or } t = 2. \end{cases}$$
\end{lemma}

\begin{proof} Let $x_{i_0}$ be the parent of $x_{i_1}$ if it exists; and let $P = \{x_{i_0} \dots, x_{i_t}\}$.  Furthermore, for $j=0,\ldots,t-1$, let $\Gamma_j$  be the induced subtree of $\Gamma$ rooted at $x_{i_j}$, and let $\Delta_j = \Gamma_j \backslash \Gamma_{j+1}$. Observe, as before, that the minimal generators of the ideals $I_t(\Gamma \backslash P)$, $(x_{i_0})$, and $I_{t-j}(\Delta_j \backslash P)$ involve pairwise disjoint sets of vertices. Thus, by Lemma \ref{lem.subtree}, the minimal free resolution of $S/(I_t(\Gamma \backslash x_{i_t}): (x_{i_1} \cdots x_{i_t}))$ is the tensor product of the minimal free resolutions of $S/I_t(\Gamma \backslash P)$, $S/(x_{i_0})$, and $S/I_{t-j}(\Delta_j \backslash P)$ for $j=0, \dots, t-1$. Therefore, the contribution to $\beta_{i,i}(S/(I_t(\Gamma \backslash x_{i_t}): (x_{i_1} \cdots x_{i_t})))$ comes from $\beta_{i,i}(S/x_{i_0} \otimes S/I_1(\Delta_{t-1} \backslash P)) = \beta_{i,i}(S/((x_{i_0})+I_1(\Delta_{t-1} \backslash P)))$.

Observe that the number of vertices of $\Delta_{t-1} \backslash P$ is $\deg(x_{i_{t-1}}) - 2$ if $t > 2$ or $t = 2$ and $h > t-1$ (i.e., when $x_{i_{t-2}}$ exists). If $t=2$ and $h=t-1$ (i.e., when $x_{i_{t-2}}$ does not exist) then the number of vertices of $\Delta_{t-1} \backslash P$ is $\deg(x_{i_{t-1}})-1$. Observe further that $x_{i_0}$ exists if $h > t-1$. Hence, the conclusion follows from the fact that the minimal free resolution of $S/(x_{i_0}+I_1(\Delta_{t-1} \backslash P))$ is the Koszul complex.
\end{proof}

\begin{theorem} \label{thm.linearstrand}
Let $\Gamma$ be a rooted tree over the vertex set $V$. Then for $i \ge 1$,
$$\beta_{i,i+t}(I_t(\Gamma)) = \begin{cases}
{\displaystyle \sum_{\substack{v \in V}} {\deg_\Gamma(v) \choose i+1}} & \text{if  } t = 2 \\
{\displaystyle \sum_{\substack{\level(v) \ge t-1}} {\deg_\Gamma(v) \choose i+1} + \sum_{\substack{\level(v) = t-2}} {\deg_\Gamma(v) - 1 \choose i+1}} & \text{if  } t > 2.\end{cases}$$
\end{theorem}

\begin{proof} Let $h = \height \Gamma$. The assertion is vacuously true if $h < t-1$ so we may assume that $h \ge t-1$. We shall use induction on $n$, the number of vertices in $\Gamma$. Again, the assertion is vacuously true if $n = t$. Assume that $n > t$. Let $x_{i_t}$ be a leaf at the highest level in $\Gamma$, and let $x_{i_1}, \dots, x_{i_t}$ be the unique path of length $(t-1)$ terminating at $x_{i_t}$. For simplicity, let $\Gamma' = \Gamma \backslash x_{i_t}$. As observed before, $\beta_{i,i+t}(I_t(\Gamma)) = \beta_{i+1,i+t}(S/I_t(\Gamma))$. By Theorem \ref{nocancellation}, we have
$$\beta_{i+1,i+t}(S/I_t(\Gamma)) = \beta_{i+1,i+t}(S/I_t(\Gamma')) + \beta_{i,i}(S/(I_t(\Gamma \backslash x_{i_t}) : (x_{i_1} \cdots x_{i_t}))).$$

If $t = 2$ then by the induction hypothesis and Lemma \ref{lem.ls1}, we have
$$\beta_{i+1,i+2}(S/I_2(\Gamma)) = \sum_{v \in V \backslash x_{i_2}} {\deg_{\Gamma'}(v) \choose i+1} + {\deg_\Gamma(x_{i_1}) - 1 \choose i}.$$
Since $\deg_{\Gamma'}(v) = \deg_\Gamma(v)$ for all $v \not\in \{x_{i_1}, x_{i_2}\}$, $\deg_{\Gamma'}(x_{i_1}) = \deg_\Gamma(x_{i_1}) - 1$, and $\deg_\Gamma(x_{i_2}) = 1$, we have
\begin{align*}
\beta_{i+1,i+2}(S/I_2(\Gamma)) & = \sum_{\substack{v \not= x_{i_1} \\ v \not= x_{i_2}}} {\deg_{\Gamma'}(v) \choose i+1} + {\deg_{\Gamma'}(x_{i_1}) \choose i+1} + {\deg_\Gamma(x_{i_1}) - 1 \choose i} \\
& = \sum_{\substack{v \not= x_{i_1} \\ v \not= x_{i_2}}} {\deg_{\Gamma}(v) \choose i+1} + {\deg_{\Gamma}(x_{i_1}) - 1 \choose i+1} + {\deg_\Gamma(x_{i_1}) - 1 \choose i} \\
& = \sum_{\substack{v \not= x_{i_1} \\ v \not= x_{i_2}}} {\deg_{\Gamma}(v) \choose i+1} + {\deg_{\Gamma}(x_{i_1}) \choose i+1} \\
& = \sum_{v \not= x_{i_1}} {\deg_\Gamma(v) \choose i+1} = \sum_{v \in V} {\deg_{\Gamma}(v) \choose i+1}.
\end{align*}
Here, the last equality follows by adding $0 = {1 \choose i+1}$.

Assume now that $t > 2$. Consider the case when $h = t-1$. By the induction hypothesis and Lemma \ref{lem.ls1}, we have
\begin{align*}
\beta_{i+1,i+t}(S/I_t(\Gamma)) & = \beta_{i+1,i+t}(S/I_t(\Gamma')) + {\deg_\Gamma(x_{i_{t-1}})-2 \choose i} \\
& = \sum_{\substack{v \not= x_{i_t} \\ \level(v) \ge t-1}} {\deg_{\Gamma'}(v) \choose i+1} + \sum_{\substack{v \not= x_{i_t} \\ \level(v) = t-2}} {\deg_{\Gamma'}(v) - 1 \choose i+1} + {\deg_\Gamma(x_{i_{t-1}})-2 \choose i}.
\end{align*}
Since $h = t-1$ and $x_{i_t}$ is at the highest level, we have $\level(x_{i_{t-1}}) = t-2$. Also, $\deg_{\Gamma'}(v) = \deg_\Gamma(v)$ for all $v \not\in \{x_{i_{t-1}}, x_{i_t}\}$ and $\deg_{\Gamma'}(x_{i_{t-1}}) = \deg_\Gamma(x_{i_{t-1}})-1$. Thus,
\begin{align*}
\beta_{i+1,i+t}(S/I_t(\Gamma)) & = \sum_{\substack{v \not= x_{i_t} \\ \level(v) \ge t-1}} {\deg_{\Gamma'}(v) \choose i+1} + \sum_{\substack{v \not= x_{i_t}, v \not= x_{i_{t-1}} \\ \level(v) = t-2}} {\deg_{\Gamma'}(v) - 1 \choose i+1} \\
& \quad \quad + {\deg_{\Gamma'}(x_{i_{t-1}}) - 1 \choose i+1} + {\deg_\Gamma(x_{i_{t-1}})-2 \choose i} \\
& = \sum_{\substack{v \not= x_{i_t} \\ \level(v) \ge t-1}} {\deg_{\Gamma}(v) \choose i+1} + \sum_{\substack{v \not= x_{i_t}, v \not= x_{i_{t-1}} \\ \level(v) = t-2}} {\deg_{\Gamma}(v) - 1 \choose i+1} \\
& \quad \quad + {\deg_{\Gamma}(x_{i_{t-1}}) -2 \choose i+1} + {\deg_\Gamma(x_{i_{t-1}})-2 \choose i} \\
& = \sum_{\substack{v \not= x_{i_t} \\ \level(v) \ge t-1}} {\deg_{\Gamma}(v) \choose i+1} + \sum_{\substack{v \not= x_{i_t}, v \not= x_{i_{t-1}} \\ \level(v) = t-2}} {\deg_{\Gamma}(v) - 1 \choose i+1} + {\deg_\Gamma(x_{i_{t-1}})-1 \choose i+1}.
\end{align*}
As before, we can add $0 = {1 \choose i+1}$ to the sum and the assertion follows.

Now we consider the case $h > t-1$. By the induction hypothesis and Lemma \ref{lem.ls1}, we have
\begin{align*}
\beta_{i+1,i+t}(S/I_t(\Gamma)) & = \beta_{i+1,i+t}(S/I_t(\Gamma')) + {\deg_\Gamma(x_{i_{t-1}})-1 \choose i} \\
& = \sum_{\substack{v \not= x_{i_t} \\ \level(v) \ge t-1}} {\deg_{\Gamma'}(v) \choose i+1} + \sum_{\substack{v \not= x_{i_t} \\ \level(v) = t-2}} {\deg_{\Gamma'}(v) - 1 \choose i+1} + {\deg_\Gamma(x_{i_{t-1}})-1 \choose i}.
\end{align*}
Since $h > t-1$ and $\level(x_{i_t}) = h$, we have $\level(x_{i_{t-1}}) \ge t-1$. As before, $\deg_{\Gamma'}(v) = \deg_\Gamma(v)$ for all $v \not\in \{x_{i_{t-1}}, x_{i_t}\}$ and $\deg_{\Gamma'}(x_{i_{t-1}}) = \deg_\Gamma(x_{i_{t-1}})-1$. Thus, using a similar computation as above, we get
\begin{align*}
\beta_{i+1,i+t}(S/I_t(\Gamma)) & = \sum_{\substack{v \not= x_{i_t}, v \not= x_{i_{t-1}} \\ \level(v) \ge t-1}} {\deg_{\Gamma'}(v) \choose i+1} + \sum_{\level(v) = t-2} {\deg_{\Gamma'}(v)-1 \choose i+1} \\
& \quad \quad + {\deg_{\Gamma'}(x_{i_{t-1}}) \choose i+1} + {\deg_\Gamma(x_{i_{t-1}}) - 1 \choose i} \\
& = \sum_{\level(v) \ge t-1} {\deg_{\Gamma}(v) \choose i+1} + \sum_{\level(v) = t-2} {\deg_{\Gamma}(v) - 1 \choose i+1}.
\end{align*}
\end{proof}

The rest of the section is devoted to classifying all rooted trees $\Gamma$ for which $I_t(\Gamma)$ has a linear resolution. Recall that $I_t(\Gamma)$ has a linear resolution if and only if $\beta_{i,j}(I_t(\Gamma)) = 0$ for all $j \not= i+t$. Our characterization is based on the following special class of rooted trees.

\begin{definition} \label{defn.broom}
A {\it broom graph} of type $t$ consists of a {\it handle}, which is a directed path $x_0, \dots, x_s$, such that every edge in $\Gamma$ (not on the handle) has the form $(x_i, y)$ for some $i \ge s-t$.
\end{definition}

    \begin{minipage}[h]{1.5in}
            \hspace*{10ex}\includegraphics{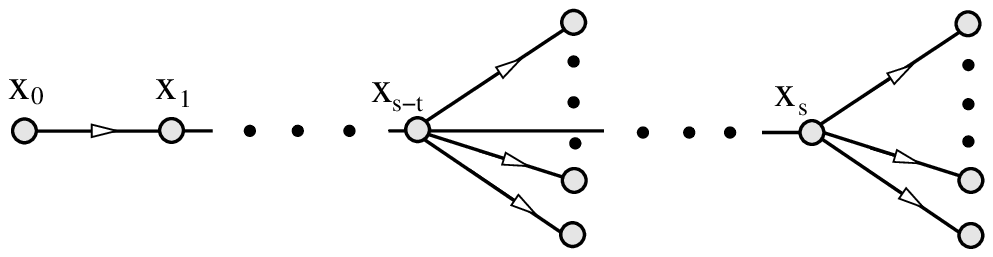}
    \end{minipage}

Since $t$ is usually fixed in our context, we often omit the phrase ``of type $t$'' and refer to a broom graph of type $t$ simply as a broom graph. Note also that in a broom graph as defined $s$ is related to the height of the graph.

\begin{remark} If $x$ is a leaf of level strictly less than $(t-1)$, then $I_t(\Gamma \backslash x)$ and $I_t(\Gamma)$ have the same generators (in different polynomial rings). This implies that $\beta_{i,j}(I_t(\Gamma \backslash x)) = \beta_{i,j}(I_t(\Gamma))$ for all $i,j$. Thus we can successively remove leaves at level strictly less than $(t-1)$ from a rooted tree without changing the graded Betti numbers of its path ideal. We call this process the {\it cleaning} process of $\Gamma$. The rooted tree obtained after the cleaning process is called the {\it clean form} of $\Gamma$, denoted $C(\Gamma)$.
\end{remark}

We are now ready to state our characterization.

\begin{theorem} \label{thm.Ntp}
Let $\Gamma$ be a rooted tree of height $h \ge t-1$. Then the following are equivalent:
\begin{enumerate}
\item $I_t(\Gamma)$ has linear first syzygies.
\item $I_t(\Gamma)$ has a linear resolution.
\item $C(\Gamma)$ is a broom graph of height at most $(2t-1)$.
\end{enumerate}
\end{theorem}

\begin{proof} Without loss of generality, we may assume that $\Gamma$ is already in its clean form, ie., $C(\Gamma) = \Gamma$. In this case, all leaves of $\Gamma$ have level at least $(t-1)$.

We shall use induction on $n$, the number of vertices in $\Gamma$. The statement is clearly true for $n \le t$ (since in this case $I_t(\Gamma)$ is either $(0)$ or $(x_1 \cdots x_t)$). Assume that $n > t$. As before, let $x_{i_t}$ be a leaf at the highest level in $\Gamma$, and let $x_{i_1}, \dots, x_{i_t}$ be the unique path of length $(t-1)$ terminating at $x_{i_t}$. Let $x_{i_0}$ be the parent of $x_{i_1}$ if it exists; and let $P = \{x_{i_0}, \dots, x_{i_t}\}$.  Furthermore, for $j=0,\ldots,t-1$, let $\Gamma_j$ be the induced subtree of $\Gamma$ rooted at $x_{i_j}$, let $\Delta_j = \Gamma_j \backslash \Gamma_{j+1}$, and let $\Gamma' = \Gamma \backslash x_{i_t}$. By Theorem \ref{nocancellation}, we have
$$\beta_{i+1,j}(S/I_t(\Gamma)) = \beta_{i+1,j}(S/I_t(\Gamma')) + \beta_{i,j-t}(S/(I_t(\Gamma'):(x_{i_1} \cdots x_{i_t}))).$$
This implies that $I_t(\Gamma)$ has linear first syzygies (resp., has a linear resolution) if and only if $I_t(\Gamma')$ has linear first syzygies (respectively, has a linear resolution), and $I_t(\Gamma'): (x_{i_1} \cdots x_{i_t})$ is generated in degree one (respectively, is generated in degree one and has a linear resolution).

By Lemma \ref{lem.subtree} the minimal free resolution of $S/(I_t(\Gamma'):(x_{i_1} \cdots x_{i_t}))$ is the tensor product of the minimal free resolutions of $S/I_t(\Gamma \backslash P)$, $S/(x_{i_0})$, and $S/I_{t-j}(\Delta_j \backslash P)$. It follows that $I_t(\Gamma'): (x_{i_1} \cdots x_{i_t})$ is generated in degree one if and only if $I_t(\Gamma \backslash P) = (0)$ and $I_{t-j}(\Delta_j \backslash P) = (0)$ for all $j < t-1$. Note that if this is the case then $I_t(\Gamma'): (x_{i_1} \cdots x_{i_t})$ has a linear resolution. Thus, (1) and (2) are equivalent.

Now we show that (2) $\Rightarrow$ (3). Since $\Gamma$ is in its clean form, if $I_t(\Gamma \backslash P) = (0)$ then $\Gamma \backslash \Gamma_0$ must be a directed path from the root of $\Gamma$ to the parent of $x_{i_0}$ of length strictly less than $(t-1)$. To show that $\Gamma$ is a broom, it now suffices to show that for any $j$, $\Delta_j \backslash P$ consists of isolated vertices. Suppose for some $j < t-1$, $\Delta_j \backslash P$ contains a path of length at least 1. In this case, we can find a path of length at least 2 of the form $x_{i_j},y_1, \dots ,y_s$ in $\Delta_j$ terminating at a leaf $y_s$ for some $2 \le s \le t-j$. Here, the second inequality is due to the fact that $x_{i_t}$ is of highest level. Recall that for $I_t(\Gamma)$ to have a linear resolution, $I_t(\Gamma') = I_t(\Gamma \backslash x_{i_t})$ must also have a linear resolution. Thus by induction and successively removing vertices, we will reduce $\Gamma$ to a rooted tree $\Gamma''$ in which $x_{i_{j+s}}$ is a leaf at the highest level and $I_t(\Gamma'')$ has a linear resolution. Let $\Delta_{t-s}''$ be the graph rooted at $x_{i_j}$ obtained from $\Gamma''$ in the same fashion as how $\Delta_j$ was obtained from $\Gamma$ (with $x_{i_{j+s}}$ replacing the role of $x_{i_t}$). Let $P''$ be the set of vertices on the unique path of length $(t-1)$ terminating at $x_{i_{j+s}}$ (this path exists since the level of $y_s$ is at least $(t-1)$). By a similar argument as with $\Delta_j$, for $I_t(\Gamma'')$ to have a linear resolution we must have $I_s(\Delta_{t-s}'' \backslash P'') = (0)$. However, this is not true since $y_1 \cdots y_s \in I_s(\Delta_{t-s}'' \backslash P'')$. We have now shown that for each $j < t-1$, $\Delta_j \backslash P$ consists of isolated vertices. Since $x_{i_t}$ is of highest level, $\Delta_{t-1} \backslash P$ also consists of isolated vertices. We can conclude that if $I_t(\Gamma)$ has a linear resolution, then $\Gamma$ is a broom of height at most $(2t-1)$.

Conversely, suppose that $\Gamma$ is a broom of height at most $(2t-1)$ in its clean form. By definition, it is easy to see that in this case, $\Gamma \backslash P$ consists of a path of length at most $(t-2)$ along with isolated vertices, and $\Delta_j \backslash P$ consists of isolated vertices for any $j$. Thus, $I_t(\Gamma \backslash P) = (0)$, $I_{t-j}(\Delta_j \backslash P) = (0)$ for all $j < t-1$, and $I_1(\Delta_{t-1} \backslash P)$ has a linear resolution. Moreover, since $\Gamma \backslash x_{i_t}$ is also a broom of height at most $(2t-1)$, by the induction hypothesis, $I_t(\Gamma \backslash x_{i_t})$ has a linear resolution. Thus, $I_t(\Gamma)$ has a linear resolution. Therefore, (3) $\Rightarrow$ (2).
\end{proof}

\begin{remark}  For any rooted tree $\Gamma$, a cellular complex supporting the linear strand of $I_t(\Gamma)$ is described by the minimal generators of $I_t(\Gamma)$; specifically, if $M_1, \dots, M_u$ are minimal generators of $I_t(\Gamma)$, then $\{M_1, \dots, M_u\}$ form a cell if and only if it is maximal with respect to the property that $$\deg\big(\operatorname{gcd}(M_1, \dots, M_u)\big) = t-1.$$ When $\Gamma$ is a broom graph of height at most $(2t-1)$, this cellular complex also supports the minimal free resolution of $I_t(\Gamma)$. In this case, set $D=\max\{\deg_\Gamma(x) ~|~ x\in\Gamma\}$ and let $v$ be a vertex of degree $D$ with highest level. Then it follows from Theorem~\ref{thm.linearstrand} that this cellular complex supporting the minimal free resolution of $I_t(\Gamma)$ has dimension (equivalently, the projective dimension of $I_t(\Gamma)$) equal to
$$\begin{cases}
        D-2 & \mbox{ if }t > 2 \mbox{ and } \level(v)=(t-2)\\
        D-1 & \mbox{ otherwise.}
    \end{cases}$$ 
\end{remark}


\section{Specialization to Path Graphs}\label{line_graphs}

In this last section of the paper, we restrict our attention to a simple class of rooted trees, namely path graphs. The \emph{path graph} over the vertex set $V = \{x_1, \dots, x_n\}$ is the directed tree whose directed edges (after a possible re-indexing) are $e_i = (x_i, x_{i+1})$ for $i = 1, \dots, n-1$.

    \begin{figure}[h]
        \centerline{
            \includegraphics{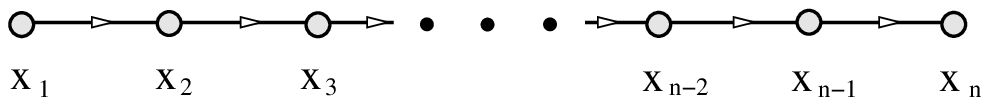}
        }
    \end{figure}

Let $L_n$ denote the path graph over $n$ vertices. Clearly,
$$I_t(L_n)=(x_1x_2\cdots x_t,~~x_2x_3\cdots x_{t+1},~~\dots,~~x_{n-t+1}x_{n-t+2}\cdots x_n).$$
It is also easy to see that
\begin{align*}
I_t(L_{n-1}):(x_{n-t+1}x_{n-t+2}\cdots x_n) =
                \begin{cases}
                (x_{n-t}) + I_t(L_{n-(t+1)}) & \text{if}~n>t\\
                0 & \text{otherwise}
                \end{cases}
\end{align*}
where we take $I_t(L_{n-(t+1)})=(0)$ if $n-(t+1)<t$.

Since the minimal resolution of $S/((x_{n-t})+I_t(L_{n-(t+1)}))$ is the tensor product of the minimal free resolutions for $S/(x_{n-t})$ and $S/I_t(L_{n-(t+1)})$, it follows from Theorem~\ref{nocancellation} that
\begin{align}
\beta_{i,j}(S/I_t(L_n)) &= \beta_{i,j}(S/I_t(L_{n-1})) + \beta_{i-1,j-t}(S/I_t(L_{n-1}):(x_{n-t+1}x_{n-t+2}\cdots x_n)) \nonumber \\
&= \beta_{i,j}(S/I_t(L_{n-1})) + \beta_{i-1,j-t}(S/I_t(L_{n-(t+1)}) + \beta_{i-2,j-t-1}(S/I_t(L_{n-(t+1)})). \label{eq.linegraph}
\end{align}

In \cite{heVanTuyl} He and Van Tuyl computed the projective dimension of $S/I_t(L_n)$. Using \eqref{eq.linegraph}, we can easily recover their formula.

\begin{corollary}\label{pdlinegraph}
Let $L_n$ be a path graph of over $n \ge t$ vertices.  Then the projective dimension of $S/I_t(L_n)$ is given by
\begin{align*}
\pd(S/I_t(L_n)) = \begin{cases}
                            \dfrac{2(n-d)}{t+1} &~\text{if}~~ n\equiv d \bmod(t+1) ~\text{for}~0\leq d\leq t-1\\
                            \dfrac{2n-(t-1)}{t+1} &~\text{if}~~ n\equiv t\bmod(t+1).
                            \end{cases}
\end{align*}
\end{corollary}

\begin{proof}  The recursive formula \eqref{eq.linegraph} gives
$$\pd(S/I_t(L_n)) = \max\{\pd(S/I_t(L_{n-1})), \pd(S/I_t(L_{n-(t+1)}))+2\}.$$
We can now proceed by using inducting on $n$.  Using the same line of arguments as in \cite[Theorem 4.1]{heVanTuyl}, the result follows.
\end{proof}

\begin{remark} \label{rmk.linegraph}
Corollary \ref{pdlinegraph}, in fact, gives us that for $n \not\equiv 0, t \bmod (t+1)$,
$$\pd(S/I_t(L_n)) = \pd(S/I_t(L_{n-1})) = \pd(S/I_t(L_{n-(t+1)}) + 2,$$
and for $n \equiv 0, t \bmod (t+1)$,
$$\pd(S/I_t(L_n)) = \pd(S/I_t(L_{n-1})) + 1 = \pd(S/I_t(L_{n-(t+1)}) + 2.$$
\end{remark}

Our last result characterizes which graded Betti numbers of $S/I_t(L_n)$ are nonzero.

\begin{theorem} \label{thm.nonzero}
Let $L_n$ be a path graph over $n \ge t$ vertices.  Then the following are equivalent:
\begin{enumerate}
\item $\beta_{i,j}(S/I_t(L_n)) \neq 0$,
\item $j-i = s(t-1)$ for some integer $s$ satisfying
$0\leq s \leq \min\{i, \left\lceil\frac{n-t+1}{t+1}\right\rceil\} \text{ and } i \leq \min \{2s, \pd(S/I_t(L_n))\}.$
\end{enumerate}
\end{theorem}

\begin{proof} The statement can be easily verified for $n = t$. We shall assume that $n > t$ and use induction on $n$.

We shall first show that (1) $\Rightarrow$ (2). To have $\beta_{i,j}(S/I_t(L_n)) \not= 0$, clearly $i \le \pd(S/I_t(L_n))$. It suffices to show that $j - i = s(t-1)$ for some $s$ satisfying $0\leq s \leq \min\{i, \left\lceil\frac{n-t+1}{t+1}\right\rceil\} \text{ and } i \leq 2s$. By the recursive formula (\ref{eq.linegraph}), we have that at least one of the graded Betti numbers $\beta_{i,j}(S/I_t(L_{n-1}))$, $\beta_{i-1,j-t}(S/I_t(L_{n-(t+1)})$ or $\beta_{i-2,j-t-1}(S/I_t(L_{n-(t+1)}))$ must be nonzero.

If $\beta_{i,j}(S/I_t(L_{n-1})) \not= 0$ then, by the induction hypothesis, we have $j - i = s(t-1)$ where $0 \le s \le \min \{i, \lceil\frac{(n-1)-t+1}{t+1}\rceil\} \le \min\{i, \lceil\frac{n-t+1}{t+1}\rceil\}$ and $i \le 2s$.  And so (2) follows.
If $\beta_{i-1,j-t}(S/I_t(L_{n-(t+1)})) \not= 0$, then by the induction hypothesis, we have $(j-i) - (t-1) = (j-t) - (i-1) = s'(t-1)$ for some $0 \le s' \le \min\{i-1, \lceil\frac{n-t+1}{t+1}\rceil - 1\}$ and $i-1 \le 2s'$. By taking $s = s'+1$, (2) again follows. In the case where $\beta_{i-2,j-t-1}(S/I_t(L_{n-(t+1)})) \not= 0$, a similar argument again implies that once again (2) follows.

We proceed to show that (2) $\Rightarrow$ (1). Suppose that $j-i = s(t-1)$ where $0 \le s \le \min\{i, \left\lceil\frac{n-t+1}{t+1}\right\rceil\}$ and $i \leq \min\{2s, \pd(S/I_t(L_n))\}$. Clearly, (1) holds if $s = 0$, so we may assume that $s > 0$.

If $s \le i-1$ (i.e., $s-1 \le i-2$), then set $s' = s-1$. By Remark~\ref{rmk.linegraph}, we have $i-2 \le \pd(S/I_t(L_{n-(t+1)}))$. Since $(j-t-1) - (i-2) = (j-i) - (t-1) = s'(t-1)$ and $i-2 \le 2s-2 = 2s'$, the induction hypothesis now implies that $\beta_{i-2,j-t-1}(S/I_t(L_{n-(t+1)})) \not= 0$, and (1) follows from \eqref{eq.linegraph}.

It remains to consider the case where $s = i$. This, in particular, implies that $s = i \le \lceil \frac{n-t+1}{t+1} \rceil$. Observe that if $n \equiv t \bmod (t+1)$, $n \ge t+1$ implies that $\lceil \frac{n-t+1}{t+1} \rceil + 1 = \lceil \frac{n+2}{t+1} \rceil \le \lceil \frac{2n - t + 1}{t+1} \rceil = \pd(S/I_t(L_n))$. On the other hand, if $n \equiv d \bmod (t+1)$ for some $d \le t-1$, then $n \ge (t+1) + d$ and we have $\lceil \frac{n-t+1}{t+1} \rceil + 1 = \lceil \frac{n+2}{t+1} \rceil \le \lceil \frac{2n - 2d}{t+1} \rceil = \pd(S/I_t(L_n))$. Therefore we have $i \le \lceil \frac{n-t+1}{t+1} \rceil \le \pd(S/I_t(L_n)) - 1 =  \pd(S/I_t(L_{n-(t+1)})) + 1$; that is, $i - 1 \le \pd(S/I_t(L_{n-(t+1)}))$. Also, $i-1 = s-1 \le 2(s-1)$. Now, set $s' = s-1$, and observe that $(j-t) - (i-1) = (j-i) - (t-1) = s'(t-1)$ and $i - 1 \le \min \{2s', \pd(S/I_t(L_{n-(t+1)}))\}$. The induction hypothesis implies that $\beta_{i-1,j-t}(S/I_t(L_{n-(n+1)})) \not= 0$, and (1) follows again from \eqref{eq.linegraph}.
\end{proof}

As a consequence of Theorem \ref{thm.nonzero}, we can compute the regularity of $S/I_t(L_n)$ and some graded Betti numbers of $I_t(L_n)$ explicitly.

\begin{corollary} \label{thm.reglinegraph}
Let $L_n$ be a path graph over $n \ge t$ vertices.  Then the Castelnuovo-Mumford regularity of $S/I_t(L_n)$ is given by
$$\reg(S/I_t(L_n)) = (t-1) \left\lceil \frac{n-t+1}{t+1} \right\rceil.$$
\end{corollary}

\begin{proof} The conclusion follows from Theorem \ref{thm.nonzero} noticing, by Corollary \ref{pdlinegraph}, that $\lceil \frac{n-t+1}{t+1} \rceil \le \pd(S/I_t(L_n)).$
\end{proof}

\begin{corollary}
Let $L_n$ be a path graph over $n \ge t$ vertices.  Then
$$\beta_{i,it}(S/I_t(L_n)) = \binom{n-it+1}{i}.$$
\end{corollary}

\begin{proof} Observe first that Theorem \ref{thm.nonzero} implies that $\beta_{i-2,it-t-1}(S/I_t(L_{n-(t+1)}))=0$ since $(it-t-1)-(i-2)=(i-1)(t-1) > (i-2)(t-1)$. Now, by induction, assume that the result holds for all $L_m$ with $m<n$.  It then follows from \eqref{eq.linegraph} that
\begin{align*}
\beta_{i,it}(S/I_t(L_n))
&= \beta_{i,it}(S/I_t(L_{n-1})) + \beta_{i-1,it-t}(S/I_t(L_{n-(t+1)}))\\
&= \binom{(n-1)-it+1}{i} + \binom{(n-(t+1))-(i-1)t+1}{i-1}\\
&= \binom{n-it}{i} + \binom{n-it}{i-1}\\
&= \binom{n-it+1}{i}.
\end{align*}
\end{proof}

\begin{remark} As before, the linear strand of $I_t(L_n)$ is supported by a cellular complex. It is not hard to see that this complex is also a path graph, namely, the path graph $L_{n-t+1}$.
\end{remark}


\end{document}